\documentclass[12pt,reqno]{amsart}
\usepackage{amsmath}
\usepackage{amssymb}
\usepackage{amstext}
\usepackage{a4wide}
\usepackage{graphicx}
\allowdisplaybreaks \numberwithin{equation}{section}
\usepackage{color}
\usepackage{cases}

\numberwithin{equation}{section}

\newtheorem{theorem}{Theorem}[section]

\newtheorem{lemma}[theorem]{Lemma}
\newtheorem*{Yudovich's Theorem}{Yudovich's Theorem}

\theoremstyle{definition}

\newtheorem{definition}[theorem]{Definition}

\theoremstyle{remark}
\newtheorem{remark}[theorem]{Remark}

\begin{document}

\title
[Nonlinear stability of planar steady Euler flows]{Nonlinear stability of planar steady Euler flows associated with semistable solutions of elliptic problems}

 \author{Guodong Wang}
\address{Institute for Advanced Study in Mathematics, Harbin Institute of Technology, Harbin 150001, P.R. China}
\email{wangguodong@hit.edu.cn}


\begin{abstract}
This paper is devoted to the study of nonlinear stability of steady  incompressible Euler flows in two dimensions. We prove that a steady Euler flow is nonlinearly stable in $L^p$ norm of the vorticity if its stream function is a semistable solution of some semilinear elliptic problem with nondecreasing nonlinearity. The idea of the proof is to show that such a flow has strict local maximum energy among flows whose vorticities are rearrangements of a given function, with the help of an improved version of Wolansky and Ghil's stability theorem. The result can be regarded as an extension of Arnol'd's second stability theorem.
\end{abstract}

\maketitle
\section{Introduction and main result}
\subsection{2D Euler equation}
Let us start by considering the following two-dimensional (2D) Euler equation governing the motion of an inviscid homogeneous fluid
\begin{equation}\label{euler}
\begin{cases}
\partial_t\mathbf v+(\mathbf v\cdot\nabla)\mathbf v=-\nabla P,&  t>0, x\in D,\\
\nabla \cdot \mathbf v=0,
\end{cases}
\end{equation}
where $D\subset\mathbb R^2$ is a bounded  simply-connected  domain with a smooth boundary, $\mathbf v=(v^1,v^2)$ is the velocity field, and $P$ is the pressure. For boundary condition, we impose  
\begin{align}\label{bc}
\mathbf v\cdot\mathbf n=0, \quad  x\in\partial D,
\end{align}
meaning that there is no mass flow across $\partial D$.

Introduce the scalar vorticity 
$\omega=\partial_{x^1}v^2-\partial_{x^2}v^1,$ the signed magnitude of the vorticity vector curl$\mathbf v$. Below we show that the Euler equation \eqref{euler} has an equivalent vorticity form. 
First by taking the curl on both sides of the first equation of \eqref{euler} we get
\begin{equation}\label{v1}
\partial_t\omega+\mathbf v\cdot\nabla\omega=0.
\end{equation}
On the other hand, the divergence-free condition $\nabla\cdot\mathbf v=0$ ensures the existence of some function $\psi$, called the stream function, such that
\[\mathbf v=\nabla^\perp\psi,\]
where $\nabla^\perp\psi=(\partial_{x^2}\psi,-\partial_{x^1}\psi)$, i.e., the clockwise rotation of $\nabla\psi$ through $\pi/2$. Clearly $\psi$ and $\omega$ are related via the following Poisson's equation
\[-\Delta \psi=\omega.\]
Moreover, the impermeability boundary condition \eqref{bc} implies that  $\psi$ is a constant  on $\partial D$ (recall that $D$ is simply-connected). Without loss of generality, we always assume that the stream function vanishes on $\partial D$. 
Thus $\psi$ is uniquely determined by $\omega$ as follows
\[\psi=\mathcal G\omega,\]
where $\mathcal G$ is the inverse of $-\Delta$ in $D$ with zero Dirichlet boundary condition. To summarize,  we have obtained  the vorticity form of the Euler equation 
\begin{equation}\label{vor}
\partial_t\omega+\nabla^\perp\mathcal G\omega\cdot\nabla\omega=0.
\end{equation}
In the rest of this paper we will mainly be focusing on the vorticity equation \eqref{vor}. 

Global well-posedness  of the initial value problem for \eqref{vor} has been studied by many authors in various function spaces in the past few decades. See \cite{De, DM, Wo, Y} for example. Among these results, a very important and interesting case is when the initial vorticity is a bounded function, for which a satisfactory existence and uniqueness result has been established  by Yudovich \cite{Y} in the 1960s. Of course, in this setting the vorticity can be discontinuous,  so we need to interpret \eqref{vor} in the weak sense. See \eqref{tweak} below.

The following version of Yudovich's result can be found in Burton's paper \cite{B5}.
\begin{Yudovich's Theorem}\label{A}
For any $\omega_0\in L^\infty(D)$, there exists a unique weak solution $\omega \in L^\infty((0,+\infty)\times D)$ such that
\begin{equation}\label{tweak}
  \int_D\omega_0(x)\zeta(0,x)dx+\int_0^{+\infty}\int_D\omega \left(\partial_t\zeta+\nabla^\perp \mathcal G\omega\cdot\nabla\zeta \right)dxdt=0,\quad \forall \,\zeta\in C_c^{\infty}(\mathbb R\times D).
  \end{equation}
 Moreover, this unique weak solution satisfies
\begin{itemize}
\item[(i)] $\omega\in C([0,+\infty);L^p(D))$ for any $p\in[1,+\infty)$;
\item[(ii)] $\omega(t,\cdot)\in \mathcal{R}_{\omega_0}$ for all $t> 0$, where $\mathcal R_{\omega_0}$ denotes the set of functions that are rearrangements of $\omega_0$, that is,
\[\mathcal R_{\omega_0}=\{w\in L^1_{\rm loc}(D)\mid \mathcal |\{x\in D\mid w(x)>a\}|=|\{ x\in D\mid \omega_0(x)>a\}|,\,\,\forall\,a\in\mathbb R\},\]
where $|\cdot|$ is the two-dimensional Lebesgue measure;
\item[(iii)] $E(\omega(t,\cdot))=E(\omega_0)$ for any $t>0$, where 
\[E(\omega(t,\cdot))=\frac{1}{2}\int_D|\mathbf v(t,x)|^2dx=\frac{1}{2}\int_D\omega(t,x)\mathcal G\omega(t,x)dx\] 
is the kinetic energy of the fluid at time $t$.
\end{itemize}
 \end{Yudovich's Theorem}

\subsection{Steady solution and nonlinear stability}
A weak solution to the vorticity equation is steady if it does not depend on the time variable. Therefore, $\omega$ is a steady solution if and only if
\begin{equation}\label{sso}
\nabla^\perp \mathcal G\omega\cdot\nabla\omega=0,\quad x\in D.
\end{equation}
Of course, if $\omega\in L^\infty(D)$, we need to interpret \eqref{sso} in the following weak sense
\begin{equation}\label{wsso}
\int_D\omega\nabla^\perp\mathcal G\bar\omega\cdot\nabla \xi dx=0,\quad\forall\,\xi\in C_c^\infty(D).
\end{equation}

There are many types of steady solutions, most of which are associated with solutions of semilinear elliptic equations. To see this,  suppose $\bar\psi\in C^2(\bar D)$ satisfies
\begin{equation}\label{seel}
\begin{cases}
-\Delta \bar\psi=g(\bar\psi),&x\in D,\\
\bar\psi=0,&x\in\partial D,
\end{cases}
\end{equation}
where $g\in C^1(\mathbb R).$
Then it is clear that  $\bar\omega=-\Delta \bar\psi=g(\bar\psi)$ belongs to $C^1(\bar D)$ and satisfies \eqref{sso} in the classical sense.
If $g$ is not $C^1$, a weak solution of \eqref{seel} also corresponds to a steady solution under very general assumptions.
See \cite{CW2}.

Given a steady solution $\bar\omega,$  a very natural problem is to study its stability. In this paper, we only consider \emph{nonlinear stability}, also called  stability of Lyapunov type.
Roughly speaking,  a steady solution $\bar\omega$ is said to be nonlinearly stable, if for any initial vorticity sufficiently ``close" to $\bar\omega$, the evolved vorticity remains ``close" to $\bar\omega$ for \emph{all time}. 

The precise definition of nonlinear stability is stated as follows.
\begin{definition}\label{dns}
Let $\bar \omega\in L^\infty(D)$ be a steady solution to the vorticity equation \eqref{vor},  $\mathcal S$ be a subset of $L^\infty(D)$, and $\|\cdot\|$ be a norm in the set of bounded measurable functions. 
If for any $\varepsilon>0,$ there exists $\delta>0,$ such that for any $\omega_0\in \mathcal S,$ $\|\omega_0-\bar\omega\|<\delta,$ it holds that $\|\omega(t,\cdot)-\bar\omega\|<\varepsilon$ for all $t>0$, then $\bar\omega$ is said to be nonlinearly stable in the norm $\|\cdot\|$ with respect to initial perturbations in $\mathcal S$. Here $\omega(t,\cdot)$ is the unique weak solution to the vorticity equation with initial vorticity $\omega_0.$
\end{definition}

Commonly used norms include:  
(1)  the enstrophy norm $\|\omega\|_{L^2(D)};$
(2)  or more generally, the $L^p$ norm of the vorticity $\|\omega\|_{L^p(D)}$, $1\leq p<+\infty$;
(3) the energy norm $$\|\nabla\mathcal G\omega\|_{L^2(D)}=(2E(\omega))^{1/2}.$$

\begin{remark}\label{r0}
Obviously, the large $\mathcal S$ is, the more stable $\bar \omega$ is. However, there are no direct relation between stabilities in two different norms, even one of them is stronger than the other. For example, the flow associated with the first eigenfunction of $-\Delta$ in $D$ with zero boundary condition is nonlinearly stable in $L^p$ norm of the vorticity for any $p\in[1,+\infty)$, but whether it is nonlinearly stable in the energy norm is still an open problem. See Section 4 in \cite{LZ2}.
\end{remark}

Of course, it is also meaningful to consider other 
 types of stability of 2D steady Euler flows, including instability and linear stability. We refer the reader to \cite{K,LZ1,LZ2,LZ3} and the references therein.

\subsection{Burton's stability criterion}

For a function $\bar\omega\in  L^\infty(D)$, if it is a local  maximizer or minimizer of $E$ on $\mathcal R_{\bar\omega}$, then it must be a steady solution. In fact, for any $\xi\in C_c^\infty(D)$,  define a family of area-preserving transformations $\Phi_t: D\to D$ 
\[\begin{cases}
\frac{d\Phi_{t}(x)}{dt}=\nabla^\perp \xi(x),&t\in\mathbb R,\\
\Phi_0(x)=x,&x\in D.
\end{cases}
\]
Then $\bar\omega(\Phi_{-t}(\cdot)\in \mathcal R_{\bar\omega}$ for all $t\in\mathbb R,$ thus $t=0$ is a local maximum or minimum point of $E(\bar\omega(\Phi_{-t}(\cdot))$, which implies
\begin{equation}\label{evar}
\frac{d}{dt}E(\bar\omega(\Phi_{-t}(\cdot))\bigg|_{t=0}=0.
\end{equation}
After some calculations we get  \eqref{wsso} from \eqref{evar}.

In 2005, Burton gave a sufficient condition for a steady solution to be nonlinearly stable, showing that if $\bar\omega$ is an \emph{isolated} local maximizer or minimizer of $E$ on $\mathcal R_{\bar\omega}$, then nonlinear stability holds. 
In Section 2, we will see that many important stability results in the literature, including Arnol'd's first and second stability theorems and Wolansky and Ghil's stability theorem,   can be put into this setting.
 \begin{theorem}[Burton, \cite{B5}]\label{bcr}
Let $p\in(1,+\infty)$ and $\bar\omega\in L^\infty(D)$ be a steady solution. Suppose that $\bar\omega$ is an isolated local maximizer  (minimizer) of $E$ over $\mathcal R_{\bar\omega}$, i.e., there exists $\delta>0$ such that 
\begin{equation}\label{l1lq}
E(\omega)<(>)E(\bar\omega), \quad\forall\,\omega\in \mathcal R_{\bar\omega},\,0<\|\omega-\bar\omega\|_{L^1(D)}<\delta,
\end{equation}
then $\bar\omega $ is nonlinearly stable in the $L^p$ norm of the vorticity with respect to initial perturbations in $L^\infty(D)$. \end{theorem}
 \begin{remark}
 The $L^1$ norm in \eqref{l1lq} can be replaced by $L^q$ norm for any $q\in[1,+\infty)$. In fact,  since $\mathcal R_{\bar\omega}$ is obviously a bounded subset of $L^\infty(D)$, the $L^1$ norm and $L^q$ norm are two equivalent norms on $\mathcal R_{\bar\omega}$ for any $q\in[1,+\infty).$
\end{remark}

The case of local minimizers of $E$ on the rearrangement class of some given function $\tilde\omega$ is simple. In fact, by Burton and McLeod \cite{BM},  
\begin{itemize}
\item If $\tilde\omega$ is two-signed, there is no local minimizer of $E$ on $\mathcal R_{\tilde\omega}$;
\item  If $\tilde\omega$ is one-signed, any local minimizer of $E$ on $\mathcal R_{\tilde\omega}$ must be a global minimizer;
\item If $\tilde\omega$ is one-signed, there exists exactly one global minimizer of $E$ on $\mathcal R_{\tilde\omega}$.

\end{itemize}
As to local maximizers, the situation becomes quite complicated. In the literature, a large class steady solutions were obtained by solving a certain variational problem of the vorticity. See \cite{BG, B1,B2,B3,CWCV,CWS,CWN,T}. In some cases, it can be checked that those solutions are local or global maximizers 
of the kinetic energy on some rearrangement class. See \cite{BG,CWCV,CWS,CWN,T}.
However, it is usually hard to check whether they are \emph {isolated}. 
In this paper, one main task is to verify that under the assumptions of Lemma \ref{lem1}, $\bar\omega$ is an isolated local maximizer of $E$ on $\mathcal R_{\bar\omega}.$

\subsection{Semistable solutions of elliptic equations}
Our aim is this paper is to study the nonlinear stability of steady Euler flows whose stream functions solve the following elliptic problem
  \begin{equation}\label{pesp1}
 \begin{cases}
-\Delta \psi=g(\psi),&x\in D,\\
\psi=0,&x\in\partial D.
\end{cases}
 \end{equation}
We will only be focusing on a particular class of solutions, i.e., \emph{semistable solutions} of \eqref{pesp1}. A solution $\bar\psi$ of \eqref{pesp1} is called a semistable solution if the linearized operator $-\Delta-g'(\bar\psi)$ of \eqref{pesp1} at $\bar\psi$ is nonnegative definite, i.e., 
\begin{equation}\label{sscs}
\int_D|\nabla \phi|^2 dx-\int_Dg'(\bar\psi)\phi^2 dx\geq 0,\,\forall\, \phi\in   H^1_0(D).
\end{equation}

Semistable solutions are a very common class of solutions of elliptic equations.
For example, if $\bar\psi\in H^1_0(D)$ is a local minimizer of the  functional
\begin{equation}\label{gfu}
\mathcal E(\psi)=\frac{1}{2}\int_D|\nabla \psi|^2dx-\int_DG(\psi)dx
\end{equation}
under every small perturbation in $C_c^\infty(D)$, where $G\in C^1(\mathbb R)$ such that $G'(s)=g(s)$ for $s\in[m,M]$, then $\bar\psi$ satisfies \eqref{pesp1} and \eqref{sscs}. This can be verified by computing the first and second variations of $\mathcal E$ at $\bar\psi$.
For a detailed discussion on more types of semistable solutions, we refer the interested reader to  \cite{Ca,CC}.

Note that \eqref{sscs} is equivalent to the following condition in terms of the vorticity
\begin{equation}\label{sscvyy}
\int_D\phi\mathcal G\phi dx-\int_Dg'(\mathcal G\bar\omega)\mathcal (G\phi)^2 dx\geq 0,\,\forall\, \phi\in   L^2(D),
\end{equation}
where $\bar\omega=-\Delta\bar\psi.$

\subsection{Main result}
Having made the above preparations, we are ready to state our main result.    

    \begin{theorem}\label{thm1}
Let $p\in(1,+\infty)$ be fixed. Let $\bar\omega\in L^\infty(D)$ be a steady solution satisfying 
\begin{equation}\label{gvvv}
\bar\omega=g(\mathcal G\bar\omega)\quad\mbox{in }D,
\end{equation}
where $g$ satisfies
\begin{itemize}
    \item [(a)] $g\in C^{1,\alpha}[m,M]$ for some $\alpha\in(0,1),$ where $m=\min_{\bar D} \mathcal G\bar\omega$ and $M=\max_{\bar D} \mathcal G\bar\omega$,
    \item [(b)] $g$ is nondecreasing on $[m,M]$.
\end{itemize}
Suppose that the following semistable condition holds
\begin{equation}\label{sscv}
\int_D\phi\mathcal G\phi dx-\int_Dg'(\mathcal G\bar\omega)\mathcal (G\phi)^2 dx\geq 0,\,\forall\, \phi\in   L^2(D).
\end{equation}
 Then  $\bar\omega$ is an isolated local maximizer of $E$ on $\mathcal R_{\bar\omega}$, thus by Theorem \ref{bcr} is nonlinearly stable in the $L^p$ norm of the vorticity with respect to initial perturbations in $L^\infty(D)$.
 \end{theorem}
   \begin{remark}\label{r100}
   By elliptic regularity theory, the $\bar\omega$ in Theorem \ref{thm1} actually belongs to $C^1(\bar D)$, thus is a classical solution to the steady vorticity equation \eqref{sso}.
 \end{remark}
 
 \begin{remark}
By Burton and McLeod \cite{BM}, if $\bar\omega$ is a local maximizer of $E$ on $\mathcal R_{\bar\omega},$ then there exists some nondecreasing function $g$ such that $\bar\omega=g(\mathcal G\bar\omega)$ in $D$. Therefore, to prove the nonlinear stability of a steady solution $\bar\omega$ satisfying $\bar\omega=g(\mathcal G\bar\omega)$  in Burton's setting,  it is \emph{necessary} to require  $g$ to be nondecreasing.  \end{remark}

   \begin{remark}\label{r100}
Compared with Lemma \ref{lem1} below, we require a stronger condition on $g$, that is, $g\in C^{1,\alpha}[m,M]$ for some $\alpha\in(0,1)$. This condition is imposed to guarantee $g'(\mathcal G\bar\omega)\in C^{\alpha}(\bar D)$ so that Lemma \ref{fefu} can be applied. Lemma \ref{fefu} is  essential in the proof of Theorem \ref{thm1}.

\end{remark}

   \begin{remark}\label{r1012}
There exist steady Euler flows that are stable but do not satisfy \eqref{sscv}. A typical example is the flow related to the following Lane-Emden equation
\begin{equation}\label{leeq}
\begin{cases}
-\Delta \psi=\psi^p,&x\in D,\\
\psi>0,&x\in D,\\
\psi=0,&x\in\partial D,
\end{cases}
\end{equation}
where $p\in(1,+\infty).$ Existence of a solution to \eqref{leeq} can be proved by applying the Nehari manifold method. It is clear that any solution $\psi$ of \eqref{leeq} satisfies 
\[\int_D|\nabla \psi|^2 dx-p\int_D\psi^{p+1}dx=(1-p)\int_D\psi^{p+1}dx<0,\]
hence there is no semistable solution to \eqref{leeq}. However, it has been proved in Theorem 1.11 of \cite{WGLA} that when $D$ is convex the flow related to any least energy solution of \eqref{leeq} (which always exists) must be stable. 

\end{remark}

Now we compare Theorem \ref{thm1} with several closely related stability results in the literature.
In the 1960s, Arnol'd \cite{A1,A2} proposed what is now usually called the energy-Casimir method and used it to prove the famous Arnol'd's first and second stability theorems for 2D steady Euler flows. Arnol'd's second stability theorem asserts that if there exist two sufficiently small positive numbers $c_1, c_2$ such that
\begin{equation}\label{tstrs}
0< c_1\leq g'(\mathcal G\bar\omega)\leq c_2 \quad \mbox{ in } D,
\end{equation}
 then the flow is nonlinearly stable in the enstrophy norm, i.e., the $L^2$ norm of the vorticity. However, since the condition \eqref{tstrs} is too strong, the application of  Arnol'd's second stability theorem is very limited . As we will see in Section 3, Arnol'd's result can only be used to prove nonlinear stability of flows with \emph{global} maximum kinetic energy on some rearrangement class. In 1990s, by introducing the method of supporting functionals, Wolansky and Ghil \cite{WG0} showed that 
the positive definiteness of the operator $-\Delta-g'(\mathcal G\bar\omega)$ in $L^2(D)$ (or equivalently, the first eigenvalue of $-\Delta-g'(\mathcal G\bar\omega)$ in $L^2(D)$ is positive) is sufficient to ensure nonlinear stability in the enstrophy norm. In Section 3, we will see that the steady flows studied in \cite{WG0} in fact have \emph{local} maximum kinetic energy on some rearrangement class. Shortly later, Wolansky and Ghil \cite{WG} refined the supporting functional method in  \cite{WG0} so that it can be used to tackle a larger class of steady flows, allowing $-\Delta-g'(\mathcal G\bar\omega)$ to have a finite number of negative eigenvalues in $L^2(D)$; however, to get nonlinear stability they need to impose some
algebraic conditions on the operator $-\Delta-g'(\mathcal G\bar\omega)$ (see Theorem 4.1 in \cite{WG}), which in most cases are not easy to verify.  

Note that both Arnol'd's second stability theorem and and Wolansky and Ghil's results in \cite{WG0,WG} require $g'$ to have a positive lower bound. This is a strong restriction and excludes many physically interesting steady flows.  
For example, if in \eqref{gvvv} the vorticity $\bar\omega$ has compact  support, then $g$ cannot be a strictly increasing function, thus $g'$ must vanish on some interval. Two-dimensional steady Euler flows with compactly supported vorticity are very common and related existence results can be found in \cite{CLW,CPY1,CPY2,CWCV,CWS,EM,SV,T} and the references therein. Of course, whether those flows satisfy the conditions in Theorem \ref{thm1} still requires careful consideration.

Compared with Arnol'd's and Wolansky and Ghil's stability results, in Theorem \ref{thm1} we only require the first eigenvalue of $-\Delta-g'(\mathcal G\bar\omega)$ in $L^2(D)$ to be nonnegative, without any other assumption. Moreover, $g$ is only assumed to be nondecreasing, rather than the stronger condition ``$g'$ has a positive lower bound". Finally, since we make use of Burton's stability criterion (i.e., Theorem \ref{bcr}), our proof is more concise and also the stability obtained is stronger (it is in in the $L^p$ norm of the vorticity for any $p\in(1,+\infty),$ not only in the enstrophy norm).

The following lemma is crucial in the proof of Theorem \ref{thm1}. To make the statement more concise, 
for $\bar\omega\in L^\infty(D)$, denote
\[\mathcal R_{\bar\omega}-\bar\omega=\{\phi\mid \phi=\omega-\bar\omega \mbox{ for some }\omega\in \mathcal R_{\bar\omega}\}.\]
It is clear that $\mathcal R_{\bar\omega}-\bar\omega$ is a bounded subset of $L^\infty(D).$

 \begin{lemma}\label{lem1}
 Let $\bar\omega\in L^\infty(D), \bar\omega\neq 0$ be a steady solution satisfying $\bar\omega=g(\mathcal G\bar\omega)$ in $D$, where $g$ satisfies 
 \begin{itemize}
    \item [(i)] $g\in C^{1}[m,M]$, where $m=\min_{\bar D} \mathcal G\bar\omega$ and $M=\max_{\bar D} \mathcal G\bar\omega$,
    \item [(ii)] $g$ is nondecreasing on $[m,M]$.
\end{itemize}
 Suppose $\int_Dg'(\mathcal G\bar\omega)dx>0$ and there exists some $\delta>0$ such that
\begin{equation}\label{deta}\int_D\phi\mathcal G\phi dx-\int_Dg'(\mathcal G\bar\omega)(\mathcal G\phi)^2 dx+\frac{\left(\int_Dg'(\mathcal G\bar\omega)\mathcal G\phi dx\right)^2}{\int_Dg'(\mathcal G\bar\omega)dx}\geq \delta \int_D(\mathcal G\phi)^2 dx ,\,\forall \phi\in   \mathcal R_{\bar\omega}-\bar\omega.
\end{equation}
Then  $\bar\omega$ is an isolated local maximizer of $E$ on $\mathcal R_{\bar\omega}$, thus is nonlinearly stable in the $L^p$ norm of the vorticity with respect to initial perturbations in $L^\infty(D)$.
 \end{lemma}

\begin{remark}
Compared with Theorem \ref{thm1}, we require \eqref{deta} to hold only for any $\phi\in   \mathcal R_{\bar\omega}-\bar\omega$, rather than $\phi\in   L^2(D)$. 
However, this difference does not make a difference in the proof of Theorem \ref{thm1} (as we will see in Section 5,   the assumptions of Theorem \ref{thm1} ensure \eqref{deta} to hold for any $\phi\in L^2(D)$). We believe that a proper use of this difference can weaken the assumptions of  Theorem \ref{thm1}, or inspire a new proof without using Lemma \ref{lem1}  just as we do in proving Arnol'd's second stability in Section 3. Of course, this may require some deep analysis of the set $\mathcal R_{\bar\omega}-\bar\omega.$ 
\end{remark}
This paper is organized as follows. In Section 2, we give some preliminaries that are used in subsequent sections. In Section 3, we introduce the energy-Casimir method and use it to simplify the proofs of
 Arnol'd's stability theorems and Wolansky and Ghil's stability theorem based on Burton's stability criterion.  In Sections 4 and 5, we prove Lemma \ref{lem1} and Theorem \ref{thm1} respectively.

\section{Preliminaries}
We prove several lemmas in this section which will be used in Sections 3 and 4.

The first lemma is about some elementary properties of strictly decreasing functions, which is mainly used in the proof of Arnol'd's first stability theorem in Section 3.

\begin{lemma}\label{fuzhu0}
Let $q:\mathbb R\to\mathbb R$ be strictly decreasing, satisfying
 \begin{equation}\label{fz11}
\lim_{s\to+\infty}q(s)=-\infty,\quad \lim_{s\to-\infty}q(s)=+\infty.\end{equation} 
 Define $p(s)=\inf\{\tau\mid q(\tau)\leq s\}.$
Then 
\begin{itemize}
\item [(i)] $p$ is nonincreasing, thus has at most countably many discontinuities;
\item [(ii)] $p(q(s))=s, \forall\,s\in\mathbb R$;
\item [(iii)] define $P(s)=\int_0^sp(\tau)d\tau,$ then $P$ is locally Lipschitz continuous and $P'(s)=p(s)$ for a.e. $s\in\mathbb R$;
\item [(iv)] for any $r,s\in\mathbb R$, it holds that
\[P(r+s)\leq P(r)+p(r)s.\]
\end{itemize}
\end{lemma}
\begin{proof}
To prove (i), just observe that
\[\{\tau\mid q(\tau)\leq s_1\}\subset \{\tau\mid q(\tau)\leq s_2\},\quad\forall\,s_1<s_2.\]

Next we prove (ii). Let $s\in \mathbb R$ be fixed. By the definition of $p$, 
\[p(q(s))=\inf\{\tau\in\mathbb R\mid q(\tau)\leq q(s)\}.\]
On the other hand, since $q$ is strictly decreasing, it is easy to check that 
\[\{\tau\in\mathbb R\mid q(\tau)\leq q(s)\}=[s,+\infty).\]
Hence the desired result follows immediately.

Now we prove (iii). First observe that $p\in L^\infty_{\rm loc}(\mathbb R)$, hence $P$ is locally Lipschitz continuous. Choose $s_0\in\mathbb R$ such that $p$ is continuous at $s_0.$ It is clear that as $s\to s_0$ 
\[\left|\frac{P(s)-P(s_0)}{s-s_0}-p(s_0)\right|\leq\frac{1}{s-s_0}\int_{s_0}^{s}|p(\tau)-p(s_0)| d\tau\to 0.\]
This means that $P'=p$ on the set $\{s\in\mathbb R\mid p \mbox{ is continuous at } s\}.$ Taking into account (i) we have the desired result.

Finally we prove (iv). By a simple calculation we have
\[P(r+s)-P(r)-p(r)s=\int_r^{r+s}p(\tau)-p(r)d\tau\leq 0.\]
Here we used the fact that $p$ is nonincreasing.

\end{proof}

The next lemma is about nondecreasing and continuous functions. To make the statement concise, for any function $q:\mathbb R\to\mathbb R$ we use $\mathcal L_q(s)$ to denote the $s$-level set of $q$, i.e.,
\[\mathcal L_q(s)=\{\tau\in\mathbb R\mid q(\tau)= s\}.\]

\begin{lemma}\label{fuzhu}
Let $q:\mathbb R\to\mathbb R$ be continuous and nondecreasing, satisfying \begin{equation}\label{fz11}
\lim_{s\to+\infty}q(s)=+\infty,\quad \lim_{s\to-\infty}q(s)=-\infty.\end{equation} 
 Define $p(s)=\inf\{\tau\mid \tau\in \mathcal L_q(s)\}.$
Then 
\begin{itemize}
\item [(i)] $q(p(s))=s, \forall\,s\in\mathbb R$;
\item [(ii)] $p(q(s))=s$ whenever $\mathcal L_q(q(s))$ is a singleton;
\item [(iii)] $p$ is strictly increasing, thus has at most countably many discontinuities;
\item[(iv)] $p$ is continuous from the left.
\item[(v)] Define $P(s)=\int_0^sp(\tau)d\tau,$ then $P$ is locally Lipschitz continuous and $P'(s)=p(s)$ for a.e. $s\in\mathbb R$;
\end{itemize}
\end{lemma}
\begin{proof}
First we prove (i). Let $s$ be fixed. It is clear from the assumptions on $q$ that  $\mathcal L_q(s)$ is a compact interval, which implies that there exists a unique $\tau_0\in\mathcal L_q(s)$ such that $p(s)=\tau_0$. Thus $q(p(s))=q(\tau_0)=s.$

Next we prove (ii). Obviously $s\in\mathcal L_q(q(s))$. On the other hand, by (i) we have
$q(p(q(s)))=q(s)$, thus $p(q(s))\in\mathcal L_q(q(s))$. The conclusion follows from the assumption that $\mathcal L_q(q(s))$ is a singleton.

To prove (iii), suppose by contradiction there exist $s_1,s_2\in\mathbb R,$ $s_1<s_2$, such that $p(s_1)\geq p(s_2).$ Then obviously $q(p(s_1))\geq q(p(s_2))$, thus by (i) we have $s_1\geq s_2$, which is an obvious contradiction.

Now we prove (iv). Let $s_0$ be fixed. By (ii) we see that the limit $\lim_{s\to s_0^-}p(s)$ exists and $\lim_{s\to s_0^-}p(s)\leq p(s_0)$. Suppose by contradiction that $\lim_{s\to s_0^-}p(s)\leq p(s_0)-\varepsilon_0$ for some $\varepsilon_0>0$, then $\lim_{s\to s_0^-}q(p(s))\leq  q(p(s_0)-\varepsilon_0),$ i.e., $s_0\leq  q(p(s_0)-\varepsilon_0).$ Taking into account the fact that $q$ is strictly increasing, we get the following contradiction $$s_0\leq q(p(s_0)-\varepsilon_0)<q(p(s_0))=s_0.$$

Finally the proof of (v) is identical to that of (iii) in Lemma \ref{fuzhu0}.

\end{proof}

We also need the notion of Legendre transform of a convex function. We summarize what are needed later in the following lemma. The conditions imposed are not optimal, but are sufficient for our use.

\begin{lemma}\label{lt}
Suppose $q:\mathbb R\to\mathbb R$ satisfies
\begin{itemize}
\item[(1)] $q\in C^1(\mathbb R)$;
\item [(2)] $q$ is nondecreasing;
\item[(3)] there exists positive constants $c_1,c_2$ such that \[\lim_{s\to+\infty}\frac{q(s)}{s}=c_1,\quad \lim_{s\to-\infty}\frac{q(s)}{s}=c_2.\]
\end{itemize}
Denote $Q(s)=\int_0^sq(\tau)d\tau$. Define the Legendre transform $\hat Q:\mathbb R\to(-\infty,+\infty]$ of $Q$ as follows
\begin{equation}\label{dolt}
\hat Q(s)=\sup_{\tau\in\mathbb R}(\tau s-Q(\tau)).
\end{equation}
Let $p,P$ be defined in Lemma \ref{fuzhu}. Then
\begin{itemize}
\item[(i)] For each fixed $s\in\mathbb R,$ $s\tau-Q(\tau)$ is a constant on $\mathcal L_q(s),$ and $\hat Q(s)=s\tau-Q(\tau)\big|_{\tau\in \mathcal L_q(s)}$;
\item [(ii)] $\hat Q(s)+Q(\tau)\geq s\tau$ for any $s,\tau\in\mathbb R,$ and the equality holds if and only if $q(\tau)=s;$
\item [(iii)] $\hat Q(s)=P(s)+\hat Q(0), \forall\,s\in\mathbb R$ (hence $\hat Q$ is locally Lipschitz continuous);
\item [(iv)] If $q$ is additionally strictly increasing, then $\hat Q\in C^1(\mathbb R)$ and $\hat Q'=p$. 
\end{itemize}
\end{lemma}

\begin{proof}
First we prove (i). Let $s$ be fixed. By (ii) and (iii), $\mathcal Q_s$ is a compact interval. Computing the derivative of $s\tau-Q(\tau)$ on $\mathcal L_q(s)$ we get
\[(s\tau-Q(\tau))'\big|_{\tau\in\mathcal L_q(s)}=s-q(\tau)\big|_{\tau\in\mathcal L_q(s)}=0,\] 
hence $s\tau-Q(\tau)$ is a constant on $\mathcal Q_s$. To prove $\hat Q(s)=(s\tau-Q(\tau))\big|_{\tau\in \mathcal L_q(s)}$, we choose some $\tau_0$ such that $\hat Q(s)=s\tau_0-Q(\tau_0)$. This is doable by (iii) and the continuity of $\hat Q.$ Obviously $\frac{d}{dr}(sr-Q(r))'\big|_{r=\tau_0}=0$, which yields $\tau_0\in\mathcal L_q(s).$ Hence \[\hat Q(s)=s\tau_0-Q(\tau_0)=(s\tau-Q(\tau))\big|_{\tau\in \mathcal L_q(s)}.\]

Next we prove (ii). The first part of (ii) follows from the definition of $\hat Q.$ Now we prove the second part. If $q(\tau)=s,$ by (i) we see that $\hat Q(s)=s\tau-Q(\tau)$. Conversely, if $\hat Q(s)=s\tau-Q(\tau)$, then $\frac{d}{dr}(sr-Q(r))\big|_{r=\tau}=0$, which gives $s=q(\tau).$

Now we prove (iii). Since $p(s)\in\mathcal L_q(s), \forall\,s\in\mathbb R$  (by (i) in Lemma \ref{fuzhu}), it is obvious that $\hat Q(s)=sp(s)-Q(p(s))$. Thus it suffices to show that
\[sp(s)-Q(p(s))=\int_0^sp(\tau)d\tau-Q(p(0)),\,\,\forall\,s\in\mathbb R.\]
By (iii) and the continuity of $q$, we see that  $\{q(t)\mid t\in\mathbb R\}=\mathbb R,$ thus we need only to prove
\begin{equation}\label{notp0}
q(t)p(q(t))-Q(p(q(t)))=\int_0^{q(t)}p(\tau)d\tau-Q(p(0)),\,\,\forall\,t\in \mathbb R,
\end{equation}
Using the fact that $p(q(t))$ and $t$ both belong to $\mathcal L_q(q(t))$, we have 
\[q(t)p(q(t))-Q(p(q(t)))=q(t)t-Q(t),\]
therefore \eqref{notp0} is in fact equivalent to
\begin{equation}\label{notp1}
q(t)t-Q(t)=\int_0^{q(t)}p(\tau)d\tau-Q(p(0)),\,\,\forall\,t\in \mathbb R.
\end{equation}
Denote 
\[h(t)=q(t)t-Q(t)-\int_0^{q(t)}p(\tau)d\tau+Q(p(0)).\]
 We need to prove that $h(t)=0,\forall\,t\in\mathbb R$. Observe first that $h(p(0))=0$, thus to finish the proof it suffices to show that $h'(t)=0 $ a.e. $t\in\mathbb R$ (obviously $h$ is locally Lipschitz continuous). By a simple calculation we have
\[h'(t)=q'(t)t-q'(t)p(q(t)) \mbox{ a.e. }t\in\mathbb R.\]
It is clear that $h'=0$ a.e. on the set $\{t\in\mathbb R\mid q'(t)=0\}$. On the other hand, for any $t\in\mathbb R$ such that $ q'(t)>0$, it is easy to check that $\mathcal L_q(q(t))$ must be a singleton, thus by (ii) in Lemma \ref{fuzhu} we also get $h'(t)=0.$ Therefore the proof of (iii) is finished.

Finally (iv) follows from (iii) and the fact that $p$ is continuous when $q$ is strictly increasing.
\end{proof}

In Theorem \ref{thm1} and Lemma \ref{lem1}, $g$ is only defined on the compact interval $[m,M]$. However, in order to apply the energy-Casimir method, we need a suitable extension of $g$ to the whole real line
such that its essential properties are retained. The following lemma makes such an extension possible.

\begin{lemma}\label{mM}
Let $[m, M]$ be a compact interval with $m<M, $ and $q$ be a $C^1$ and nondecreasing function defined on $[m,M].$  Then there exists a function $\tilde q:\mathbb R\to\mathbb R$ such that
\begin{itemize}
\item[(i)] $\tilde q\in C^1(\mathbb R)$, and $\tilde q(s)=q(s)$  whenever $s\in [m,M]$;
\item [(ii)] $\tilde q$ is strictly increasing in $(-\infty, m]$ and $[M,+\infty)$;
\item[(iii)] there exists positive constants $c_1,c_2$ such that \[\lim_{s\to+\infty}\frac{\tilde q(s)}{s}=c_1,\quad \lim_{s\to-\infty}\frac{\tilde q(s)}{s}=c_2.\] \end{itemize}

\end{lemma}

\begin{proof}
We only give the extension of $q$ on $(-\infty,m]$. The other half can be constructed similarly.

It is clear that $g'(m)\geq 0$.  If $g'(m)>0$, we define for each $s\in(-\infty,m]$
\begin{equation}
\tilde q(s)=
q'(m)(s-g(m)).
\end{equation}
If $g'(m)=0$, we define 
\begin{equation}
\tilde q(s)=
\begin{cases}
q(m)-(s-m)^2&\mbox{ if } m-1\leq s\leq m,\\
2(x-m+1)+q(m)-1 &\mbox{ if } s< m-1.
\end{cases}
\end{equation}
It can be verified directly that the function $\tilde q$ defined above satisfies all the  required properties. 
\end{proof}

In the proof of Theorem \ref{thm1} in Section 5, a key ingredient is to use the fact that the first eigenfunction of an elliptic operator is of constant sign.
Although this result can be found in many textbooks, we give the statement below for the reader's convenience.
\begin{lemma}\label{fefu}
Let $c\in C^\alpha(\bar D)$, and $\mu_1$ be the first eigenvalue of the elliptic operator $-\Delta+c$, i.e.,
\[\mu_1=\inf\left\{\int_D|\nabla u|^2 +cu^2dx\mid u\in H^1_0(D),\| u\|_{L^2(D)=1}\right\}.\]
If $\tilde u\in H^1_0(D),\tilde u\neq 0$ is the first eigenfunction of $-\Delta+c$, or equivalently, 
\[ \frac{\int_D|\nabla \tilde u|^2 +c\tilde u^2dx}{\int_D\tilde u^2dx}=\mu_1,\]
then either $\tilde u> 0$ or $\tilde u< 0$ in $D$.
\end{lemma}
\begin{proof}
See Chapter 6 in \cite{LCE} for example. 
\end{proof}
\begin{remark}
In Lemma \ref{fefu}, the condition $c\in C^\alpha(\bar D)$ is necessary. In fact, a crucial step in the proof of Lemma \ref{fefu} is to show that any weak solution to  
the following elliptic equation must be of $C^2$ 
\begin{equation}
\begin{cases}
-\Delta u+cu=\mu_1u,&x\in D,\\
u=0,&x\in\partial D.
\end{cases}
\end{equation}
To show this, we need to apply Schauder theory, where the coefficient $c$ is required to be of $C^\alpha.$
\end{remark}

The following lemma will be used in the proof of Arnol'd's second stability theorem in Section 3.
\begin{lemma}\label{25yela}
For any $\phi\in L^2(D)$, it holds that 
\begin{equation}\label{25la}
\int_D\phi\mathcal G\phi dx\leq\frac{1}{\lambda_1} \int_D\phi^2 dx,
\end{equation}
and the equality holds if and only if $\phi$ is an eigenfunction of $-\Delta$ in $D$ with zero boundary condition associated with the principal eigenvalue $\lambda_1.$
\end{lemma}
\begin{proof}
Let $\{\phi_k\}_{k=1}^{+\infty}$ be the orthogonal basis of $L^2(D)$ satisfying 
\[\phi_k\in H^1_0(D),\quad-\Delta \phi_k=\lambda_k \phi_k,\quad\|\phi_k\|_{L^2(D)}=1,\]
where $\{\lambda_k\}_{k=1}^{+\infty}$ is the set of eigenvalues of $-\Delta$ in $D$ with zero Dirichlet data, satisfying
\begin{equation}\label{increa}
0<\lambda_1<\lambda_2\leq\lambda_3\leq\cdot\cdot\cdot.
\end{equation}
For any $\phi\in L^2(D)$, we expand it in terms of  the basis $\{\phi_k\}_{k=1}^{+\infty}$ as follows
\[\phi=\sum_{k=1}^{+\infty}a_k\phi_k.\]
Obviously 
\begin{equation}\label{el11}
\sum_{k=1}^{+\infty}a^2_k=\|\phi\|^2_{L^2(D)}.
\end{equation}
Moreover, it is easy to check that
\[\int_D\phi\mathcal G\phi dx=\sum_{k=1}^{+\infty}\frac{a^2_k}{\lambda_k}.\]
Hence by \eqref{increa} and \eqref{el11} we have
\begin{equation}\label{0011}
\lambda_1\int_D\phi\mathcal G\phi dx=\sum_{k=1}^{+\infty}\frac{\lambda_1}{\lambda_k}a_k^2\leq \sum_{k=1}^{+\infty}\frac{\lambda_1}{\lambda_1}a_k^2= \|\phi\|^2_{L^2(D)},
\end{equation}
which is exactly \eqref{25la}.
Moreover, it is easy to see that the inequality in \eqref{0011} is an equality if and only if $a_k=0, \forall\,k\geq 2$, which means that the equality in \eqref{25la} holds if and only if $\phi=a_1\phi_1$.

\end{proof}

\section{Energy-Casimir method}

The energy-Casimir method was first proposed by Arnol'd in the 1960s and has become a very powerful tool in the stability analysis of steady solutions of infinite-dimensional Hamiltonian systems. 
In this section, we first recall this method and then show how it can be applied to prove nonlinear stability of 2D steady Euler flows. Although the results in this section have already appeared in the literature, 
we give them new and simplified proofs based on Burton's stability criterion, from which stronger stability follows.

\subsection{Energy-Casimir method and Arnol'd's stability theorems}
Let $\bar\omega\in L^\infty(D)$ be a steady solution satisfying $\bar\omega=g(\mathcal G\bar\omega)$ in $D$, where $g\in C(\mathbb R)$ is strictly increasing. To study the nonlinear stability of $\bar\omega$, Arnol'd considered a flow-invariant functional 
\begin{equation}\label{ecf}
EC(\omega)=E(\omega)-\int_D F(\omega)dx,
\end{equation}
where $F(s)=\int_0^sg^{-1}(\tau)d\tau$. 
Notice that there are two parts in the definition of $EC$, i.e., the energy part $E$ and the Casimir part $\int_DF(\omega)dx$, thus $EC$ is often called the energy-Casimir functional.
By choosing $EC$ as the Lyapunov functional and using conservative quantities of the vorticity equation,
Arnol'd showed that  if $F$ is concave (Arnol'd's first stability theorem) or $F''$ is very convex (Arnol'd's second stability theorem), then $\bar\omega$ is nonlinearly stable in the enstrophy norm.

 Below we give the precise statement of Arnol'd's first  and second stability theorems and present a simple proof based on Burton's stability criterion.  From the proof, we will see that  the steady solution in Arnol'd's first (second) stability theorem is in fact a global minimizer (maximizer) of $E$ on some rearrangement class.  

 \begin{theorem}[Arnol'd's first and second stability theorems, \cite{A1,A2}]\label{ar12}
Let $\bar\omega\in L^\infty(D)$ be a steady solution satisfying $\bar\omega=g(\mathcal G\bar\omega)$ a.e. in $D$ for some function $g:[m,M]\to\mathbb R$, where $m=\min_{\bar D}\mathcal G\bar\omega$ and $M=\max_{\bar D}\mathcal G\bar\omega$. 
\begin{itemize}
 \item[(1)] If $g$ is strictly decreasing in $[m,M],$ then $\bar\omega$ is the unique global minimizer of $E$ on $\mathcal R_{\bar\omega}$.
 \item [(2)] If $g\in C^1[m,M]$ satisfying $0<\min_{[m,M]}g'\leq \max_{[m,M]}g'\leq\lambda_1$, then $\bar\omega$ is the unique global maximizer of $E$ on $\mathcal R_{\bar\omega}$.
 \end{itemize}
 In both cases, by Theorem \ref{bcr}, $\bar\omega$ is nonlinearly stable in the  $L^p$ norm of the vorticity with respect to initial perturbations in $L^\infty(D)$, where $p\in(1,+\infty).$
\end{theorem}
\begin{remark}
It is not known whether (1) of Theorem \ref{ar12} still holds when $g$ is only nonincreasing in $[m,M]$. 
\end{remark}

\begin{proof}
We prove (1) first. Without loss of generality,  we assume that $g$ is defined on $\mathbb R$ such that $g$ is strictly decreasing and satisfies
\[\lim_{s\to+\infty}g(s)=-\infty, \quad\lim_{s\to-\infty}g(s)=+\infty.\]
Define 
\[f(s)=\inf\{\tau\mid g(\tau)\leq s\}\mbox{ and }
 F(s)=\int_0^sf(\tau)d\tau.\]
  By (ii) in Lemma \ref{fuzhu0}, we have 
 \begin{equation}\label{spinx}
 f(\bar\omega)=f(g(\mathcal G\bar\omega))=\mathcal G\bar\omega \mbox{ a.e. in }D.
 \end{equation}
Now for any $\phi\in \mathcal R_{\bar\omega}-\bar\omega$, $\phi\neq 0$ (thus $\int_D\phi\mathcal G\phi dx>0$), we compare $E(\bar\omega+\phi)$ and $E(\bar\omega)$ as follows
\begin{align*}
E(\bar\omega+\phi)-E(\bar\omega)&=EC(\bar\omega+\phi)-EC(\bar\omega)\\
&=\frac{1}{2}\int_D\phi\mathcal G\phi dx-\int_D F(\bar\omega+\phi)-F(\bar\omega)-\mathcal G\bar\omega\phi dx\\
&=\frac{1}{2}\int_D\phi\mathcal G\phi dx-\int_D F(\bar\omega+\phi)-F(\bar\omega)-f(\bar\omega)\phi dx\\
&\geq \frac{1}{2}\int_D\phi\mathcal G\phi dx\\
&>0.
\end{align*}
Note that in the first equality we used the fact that $EC$ is a constant on $\mathcal R_{\bar\omega}$ (recall by Lemma \ref{fuzhu0} that $F$ is locally Lipschitz continuous), in the third equality we used \eqref{spinx}, and in the first inequality we used
\begin{equation*}
F(\bar\omega+\phi)\leq F(\bar\omega)+f(\bar\omega)\phi \,\,\mbox{ a.e. in }D,
 \end{equation*}
which follows from (iv) in Lemma \ref{fuzhu0}.

Now we prove (2). Without loss of generality, we assume that $g$ satisfies
\begin{itemize}
    \item [(i)] $g\in C^1(\mathbb R)$ and $\lim_{s\to+\infty}g(s)=+\infty, \lim_{s\to-\infty}g(s)=-\infty$;
    \item  [(ii)]$\inf_{\mathbb R}g'> 0$;
    \item  [(iii)]$\sup_{\mathbb R}g' \leq \lambda_1$.
\end{itemize}
This can be done by repeating the proof of Lemma \ref{mM}. Obviously $g$ is strictly increasing on $\mathbb R,$ thus we can define 
its inverse function $f=g^{-1}$.  Let $F(s)=\int_0^sf(\tau)d\tau.$
According to (ii), $f\in C^1(\mathbb R),$ thus $F\in C^2(\mathbb R)$. Moreover, from (iii) we see that  $\inf_{\mathbb R}f'\geq1/\lambda_1$. Now for any $\phi\in\mathcal R_{\bar\omega}-\bar\omega, \phi\neq 0$, we have
\begin{align*}
E(\bar\omega+\phi)-E(\bar\omega)&=EC(\bar\omega+\phi)-EC(\bar\omega)\\
&=\frac{1}{2}\int_D\phi\mathcal G\phi dx-\int_D F(\bar\omega+\phi)-F(\bar\omega)-\mathcal G\bar\omega\phi dx\\
&=\frac{1}{2}\int_D\phi\mathcal G\phi dx-\int_D F(\bar\omega+\phi)-F(\bar\omega)-f(\bar\omega)\phi dx\\
&\leq \frac{1}{2}\int_D\phi\mathcal G\phi dx-\frac{1}{2}\inf_{\mathbb R}f'\int_D \phi^2 dx\\
&\leq \frac{1}{2}\left(\int_D\phi\mathcal G\phi dx-\frac{1}{\lambda_1} \int_D\phi^2 dx\right).
\end{align*}
To finish the proof, it suffices to show that
\begin{equation}\label{qbunq}
\int_D\phi\mathcal G\phi dx-\frac{1}{\lambda_1} \int_D\phi^2 dx<0.
\end{equation}
In fact, by Lemma \ref{25yela} we have 
\[\int_D\phi\mathcal G\phi dx-\frac{1}{\lambda_1} \int_D\phi^2 dx\leq 0.\]
If the equality holds, then by  Lemma \ref{25yela}  $\phi$ is an eigenfunction of $-\Delta$ associated with the principal eigenvalue $\lambda_1.$ Taking into account Lemma \ref{fefu} and the fact that $\phi\neq 0,$ we see that $\phi>0$ or $\phi<0$ in $D$. Therefore 
\[\int_D\phi dx>0 \mbox{ or }\int_D\phi dx<0. \]
On the other hand, since $\phi\in \mathcal R_{\bar\omega}-\bar\omega$, we have
\[\int_D\phi dx=0,\]
which is a contradiction. Thus \eqref{qbunq} is proved.
\end{proof}

\subsection{Wolansky and Ghil's supporting functional method}
Arnold's second stability theorem is about   \emph{global} maximizers of $E$ on some rearrangement class. For \emph{local} maximizers, one may impose the condition that the second variation of $EC$ at $\bar\omega$ is negative definite (note that by the relation $\bar\omega=g(\mathcal G\bar\omega)$ the first variation of $EC$ at $\bar\omega$ is zero), i.e.,
 \[\int_{D}\phi\mathcal G\phi dx-\int_Df'(\bar\omega)\phi^2dx\leq -\delta \int_D\phi^2dx,\,\forall \phi\,\in L^2(D)\]
 for some $\delta>0.$
 However, this does not work. The reason is that  in the vorticity space the remainder of the Taylor expansion can not be controlled properly.

To overcome this difficulty, Wolansky and Ghil  \cite{WG0} introduced the method of supporting functionals, turning the problem into one in the stream function space. Below we state their result and give a simplified proof.

\begin{theorem}[Wolansky and Ghil, \cite{WG0}]\label{wg1}
Let $\bar\omega\in L^\infty(D)$ be a steady solution satisfying $\bar\omega=g(\mathcal G\bar\omega)$ for some $g\in C^1[m,M]$, where $m=\min_{\bar D}\mathcal G\bar\omega$ and $M=\max_{\bar D}\mathcal G\bar\omega$. 
Suppose that $g$ is strictly increasing on $[m,M]$, and there exists some $\delta>0$ such that
\begin{equation}\label{ssswg}
\int_D\phi\mathcal G\phi dx-\int_Dg'(\mathcal G\bar\omega)(\mathcal G\phi)^2 dx\geq \delta \int_D(\mathcal G\phi)^2 dx,\,\forall \phi\,\in  \mathcal R_{\bar\omega}-\bar\omega.
\end{equation}
Then  $\bar\omega$ is an isolated local maximizer of $E$ on $\mathcal R_{\bar\omega}$, thus by Theorem \ref{bcr} $\bar\omega$ is nonlinearly stable in  the $L^p$ norm of the vorticity with respect to initial perturbations in $L^\infty(D)$, where $p\in(1,+\infty).$
\end{theorem}

\begin{remark}
In  Wolansky and Ghil's original statement, $g'$ is supposed to have a positive lower bound, and \eqref{ssswg} is required to hold for any $\phi\in L^2(D)$.

\end{remark}
\begin{remark}
If $\max_{\bar D}g'(\mathcal G\bar\omega)<\lambda_1,$ then clearly \eqref{ssswg} holds. However, if we only know $\max_{\bar D}g'(\mathcal G\bar\omega)\leq\lambda_1,$ we are not sure whether \eqref{ssswg} still holds, even if $g'(\mathcal G\bar\omega)\equiv \lambda_1$ in $D$.

\end{remark}

\begin{proof}We divide the proof into three steps.

{\bf Step 1}:  We extend $g$ to be a function defined on $\mathbb R$ as in the proof of Lemma \ref{mM}.   Obviously $g$ is strictly increasing.
Denote $f=g^{-1}$ and  define $F(s)=\int_0^sf(\tau)d\tau.$ Let $\hat F$ be the Legendre transform of $F$ (see \eqref{dolt}).
By (iv) of Lemma \ref{lt} (taking $q=f$), we see that  $\hat F\in C^1(\mathbb R)$ and $\hat F'=g.$ Introduce 
\[\mathcal D(\omega)=-\frac{1}{2}\int_D\omega\mathcal G\omega dx+\int_D \hat F(\mathcal G\omega) dx,\,\,\omega\in\mathcal R_{\bar\omega}.\]
We claim that $\mathcal D$ is a supporting functional of $EC$, i.e., 
\begin{itemize}
    \item [(i)]$\mathcal D(\bar\omega)=EC(\bar\omega)$;    \item [(ii)]$\mathcal D(\omega)\geq EC(\omega),\,\forall\, \omega\in \mathcal R_{\bar\omega}.$ 
\end{itemize}
To prove (i), we recall by (i) in Lemma \ref{lt} that
$\hat F(s)=sg(s)-F(g(s)), \forall\,s\in\mathbb R.$
Thus 
\begin{align*}
\mathcal D(\bar\omega)&=-\frac{1}{2}\int_D\bar\omega\mathcal G\bar\omega dx+\int_D \hat F(\mathcal G\bar\omega) dx\\
&=-\frac{1}{2}\int_D\bar\omega\mathcal G\bar\omega dx+\int_D g(\mathcal G\bar\omega)\mathcal G\bar\omega-F(g(\mathcal G\bar\omega)) dx\\
&=-\frac{1}{2}\int_D\bar\omega\mathcal G\bar\omega dx+\int_D \bar\omega\mathcal G\bar\omega-F(\bar\omega) dx\\
&=EC(\bar\omega).
\end{align*}
To prove (ii), using the definition of Legendre transform, we have for each $\omega\in \mathcal R_{\bar\omega}$
  \begin{align*}
EC(\omega)&=\frac{1}{2}\int_D\bar\omega\mathcal G\bar\omega dx-\int_D  F(\omega) dx\\
&=-\frac{1}{2}\int_D\omega\mathcal G\omega dx+\int_D \omega\mathcal G\omega-F(\omega)dx\\
&\leq-\frac{1}{2}\int_D\omega\mathcal G\omega dx+\int_D\hat F(\mathcal G\omega) dx\\
&=\mathcal D(\omega).
\end{align*}

{\bf Step 2}: We show by contradiction that $\bar\omega$ is a strict local maximizer of  $\mathcal D$ 
on $\mathcal R_{\bar\omega}$.  Suppose that there exists a sequence $\{\phi_n\}_{n=1}^{+\infty}\subset \mathcal R_{\bar\omega}-\bar\omega$ such that  $\phi_n\neq 0$ for each $n$, $\|\phi_n\|_{L^2(D)}\to 0$ as $n\to+\infty$, and 
\begin{equation*}
\mathcal D(\bar\omega+\phi_n)\geq \mathcal D(\bar\omega)\quad\mbox{ for each }n,
\end{equation*}
or equivalently,
\begin{equation*}
-\frac{1}{2}\int_D(\bar\omega+\phi_n)\mathcal G(\bar\omega+\phi_n) dx+\int_D\hat F(\mathcal G\bar\omega+\mathcal G\phi_n) dx\geq -\frac{1}{2}\int_D\bar\omega\mathcal G\bar\omega dx+\int_D\hat F(\mathcal G\bar\omega) dx\,\,\,\mbox{ for each }n.
\end{equation*}
Using the mean value theorem and the fact $\hat F'=g$ (by Lemma \ref{lt}), we get
\begin{equation}\label{exp}
\int_D\phi_n\mathcal G\phi_n dx\leq\int_D g'(\mathcal G\bar\omega+\theta_n\mathcal G\phi_n)(\mathcal G\phi_n)^2dx \quad\mbox{ for each }n,
\end{equation}
where $0\leq\theta_n\leq 1$.
Since $\phi_n\to 0$ in $L^2(D)$ and $\phi_n\in \mathcal R_{\bar\omega}-\bar\omega$, we see that $\phi_n\to 0$ in $L^q(D)$ for each $q\in[1,+\infty)$, thus by elliptic estimate we obtain as $n\to+\infty$
\begin{equation}\label{lin0}\|\mathcal G\phi_n\|_{L^\infty(D)}\to0.\end{equation}
Now \eqref{exp} and \eqref{lin0} together yield
\[\int_D\phi_n\mathcal G\phi_n dx\leq \int_D g'(\mathcal G\bar\omega)(\mathcal G\phi_n)^2dx+ o(1)\int_D (\mathcal G\phi_n)^2,\]
where $o(1)\to0$  as $n\to+\infty.$ This obviously contradicts \eqref{ssswg}.

{\bf Step 3}: By Step 1 and Step 2, we deduce that $\bar\omega$ must be a strict local maximizer of $EC$ on $\mathcal R_{\bar\omega}.$ Since the Casimir functional $\int_DF(\omega) dx$ is a constant on $\mathcal R_{\bar\omega},$ we see that $\bar\omega$ is a strict local maximizer of $E$ on $\mathcal R_{\bar\omega}.$

\end{proof}

In 1998, Wolansky and Ghil wrote another paper \cite{WG} to show that the condition \eqref{ssswg} can be relaxed further. Since the proof is very complicated, we only summarize their result as follows.

\begin{theorem}[Wolansky and Ghil, \cite{WG}]\label{wg2}
Let $\bar\omega\in L^\infty(D)$ be a steady solution satisfying $\bar\omega=g(\mathcal G\bar\omega)$ for some $g\in C^1[m,M]$, $\min_{[m,M]}g'>0$, where $m=\min_{\bar D}\mathcal G\bar\omega$ and $M=\max_{\bar D}\mathcal G\bar\omega$. 
Define
  \[W:=\left\{\psi\mid \psi=h(\mathcal G\bar\omega), h\in C^2[m,M]\right\},\]
  \[<\psi_1,\psi_2>:=\int_D\psi_1\psi_2g'(\mathcal G\bar\omega)dx.\]
 Suppose there exist some $m$-dimensional subspace $W_m$ of $W$ for some positive integer $m$, a $g'(\mathcal G\bar\omega)$ weighted orthogonal basis $\{\xi^0_1,\cdot\cdot\cdot,\xi^0_m\}$  of $W_m$,  and some $\delta>0$, such that
\begin{equation}\label{bbb}
\int_D\phi\mathcal G\phi dx-\int_Dg'(\mathcal G\bar\omega)(\mathcal G\phi)^2 dx+\sum_{i=1}^m|<\xi^0_i,\mathcal G\phi>|^2\geq \delta \int_D(\mathcal G\phi)^2 dx,\,\forall \phi\in  L^2(D).
\end{equation}
Then $\bar\omega$ is an isolated local maximizer of $E$ on $\mathcal R_{\bar\omega}$, thus by Theorem \ref{bcr} $\bar\omega$ is nonlinearly stable in  the $L^p$ norm of the vorticity with respect to initial perturbations in $L^\infty(D)$, where $p\in(1,+\infty).$.  
\end{theorem}
 \begin{remark}
The conclusion ``$\bar\omega$ is an isolated local maximizer of $E$ on $\mathcal R_{\bar\omega}$" in Theorem \ref{wg2} does not appear directly in Wolansky and Ghil's original statement.  However, this can be easily verified by checking their proof carefully.
 \end{remark}
  \begin{remark}
 The assumption $\min_{[m,M]}g'>0$ is indispensable in Theorem \ref{wg2}, since the proof requires $F\in C^2$. 
 \end{remark}
 \begin{remark}
In most cases the condition \eqref{bbb} in Theorem \ref{wg2} is not easy to check. A refined version of Theorem \ref{wg2} can be found in \cite{LZ2}, where \eqref{bbb} is replaced by a condition on the level sets of the stream function.
 \end{remark}
 
Wolansky and Ghil proved Theorem \ref{wg2} by modifying the method of supporting functionals in \cite{WG0}, with appropriate use of vorticity conservation of the vorticity equation. In the next section, we will provide a new and simplified proof of their result for $m=1$ and 
\[W_1=\{\psi\mid \psi\mbox{ is a constant in }D\}.\]

In our proof, the condition $\min_{[m,M]}g'>0$ can be replaced by a weaker one, i.e.,  $g$ is nondecreasing on $[m,M]$.

\section{Proof of Lemma \ref{lem1} }
Throughout this section we  assume that $g$ is defined on $\mathbb R$ and satisfies (i)(ii)(iii) in Lemma \ref{mM}. 

To prove Lemma \ref{lem1}, the basic idea is to modify the functional $\mathcal D$ in the proof of Theorem \ref{wg1}, such that it is more ``close" to the energy-Casimir functional. Of course, since $g$ is only nondecreasing, the inverse function of $g$ is not well-defined, and thus it is not reasonable to define the energy-Casimir functional as in \eqref{ecf} anymore. Instead, we use the following definition of $EC$ in this section
\[EC(\omega)=\frac{1}{2}\int_D\omega\mathcal G\omega dx-\int_D\hat G(\omega) dx.\]
By Lemma \ref{lt} (taking $q=g$), the above definition is in fact a generalization of the definition \eqref{ecf}.

Define a new functional $\mathcal D_\lambda$ on $\mathcal R_{\bar\omega}$
\[\mathcal D_{\lambda}(\omega)=-\frac{1}{2}\int_D\omega\mathcal G \omega dx +\int_DG(\mathcal G\omega-\lambda) dx+\lambda M_0,\]
where  $\lambda \in\mathbb R$ is a parameter and
\begin{equation}\label{m01}M_0=\int_D\bar\omega dx.\end{equation}

Note that $\mathcal D_\lambda$ is in fact a special case of the functional (3.8) introduced in \cite{WG}. A noteworthy difference is that the $M_0$ in \cite{WG} is  a parameter, instead of a fixed number determined by \eqref{m01}. This difference makes the proof in \cite{WG} more complicated. The necessity of taking $M_0$ as a parameter is to deal with perturbations that are not in $\mathcal R_{\bar\omega}.$ However, since our proof is based on Burton's stability criterion, we only need to consider perturbations in $\mathcal R_{\bar\omega},$ therefore we can take $M_0$ to be fixed. 
\begin{lemma}\label{epex}
It holds that
\begin{equation}\label{algec}
\mathcal D_\lambda(\omega)\geq EC(\omega),\quad\forall\,\lambda\in\mathbb R,\,\omega\in\mathcal R_{\bar\omega}.
\end{equation}
\begin{proof}
 For any $\lambda\in\mathbb R$ and $\omega\in\mathcal R_{\bar\omega}$, we have
\begin{align*}
EC(\omega)&=\frac{1}{2}\int_D\omega\mathcal G\omega dx-\int_D\hat G(\omega) dx\\
&=\frac{1}{2}\int_D\omega\mathcal G\omega dx-\int_D\lambda\omega+\hat G(\omega) dx+\lambda M_0\\
 &\leq \frac{1}{2}\int_D\omega\mathcal G\omega dx-\int_D\lambda\omega+\omega(\mathcal G\omega-\lambda)-G(\mathcal G\omega-\lambda)dx+\lambda M_0\\
 &=-\frac{1}{2}\int_D\omega\mathcal G\omega dx+\int_DG(\mathcal G\omega-\lambda)dx+\lambda M_0\\
&=\mathcal D_{\lambda}(\omega).
\end{align*}
Here we used  
\[\int_D\omega dx=\int_D\bar\omega dx=M_0, \quad\forall\,\omega\in\mathcal R_{\bar\omega},\]
and
\[\hat G(\omega)\geq \omega(\mathcal G\omega-\lambda)-G(\mathcal G\omega-\lambda),\quad\forall \,\lambda\in\mathbb R\]
by Lemma \ref{lt}. 
\end{proof}
\end{lemma}

\begin{lemma}\label{epex}
For each fixed $\omega\in \mathcal R_{\bar\omega}$, 
there exits  $ \bar\lambda\in\mathbb R$,
depending on $\omega$, such that 
\begin{equation}\label{ba68}
\mathcal D_{\bar\lambda}(\omega)=\min_{\lambda\in\mathbb R}\mathcal D_{\lambda}(\omega).
\end{equation}
Moreover, any such $\bar\lambda$ satisfies
\begin{equation}\label{blam}
\int_Dg(\mathcal G\omega-\bar\lambda)dx=M_0.
\end{equation}
\end{lemma}
\begin{proof}
Fix $\omega\in\mathcal R_{\bar\omega}$. It is clear that $\mathcal D_\lambda(\omega)$ is continuous with respect to $\lambda$ and satisfies
\[\lim_{|\lambda|\to+\infty}\mathcal D_\lambda(\omega)=+\infty,\]
Thus there exists some $\bar\lambda$ such that
\[\mathcal D_{\bar\lambda}(\omega)=\min_{\lambda\in\mathbb R}\mathcal D_{\lambda}(\omega).\]
Moreover, any such $\bar\lambda$ necessarily satisfies
\[\frac{d\mathcal D_{\lambda}(\omega)}{d\lambda}\bigg|_{\lambda=\bar \lambda}=0,\]
or equivalently
\[\int_Dg(\mathcal G\omega-\bar\lambda) dx=M_0.\]

\end{proof}

It is worth mentioning that for fixed $\omega\in\mathcal R_{\bar\omega},$ the $\bar\lambda$ in Lemma \ref{epex} may be not unique.  However, the next lemma shows that for $\bar\omega$, the corresponding $\bar\lambda$ is unique and is exactly 0.

\begin{lemma}\label{ba67}
 If $\bar\lambda$ satisfies $\int_Dg(\mathcal G\bar\omega-\bar\lambda) dx=M_0,
$ then $\bar\lambda=0.$
\end{lemma}
\begin{proof}
Suppose by contradiction that there exists some $\lambda\neq 0$ such that $\int_Dg(\mathcal G\bar\omega-\bar\lambda) dx=M_0.
$ Without loss of generality, assume that $\bar\lambda>0$. Then by \eqref{m01} and the relation $\bar\omega=g(\mathcal G\bar\omega)$ we have
\begin{equation}\label{ba000}
\int_Dg(\mathcal G\bar\omega)dx=\int_Dg(\mathcal G\bar\omega-\bar\lambda) dx.
\end{equation}
Since $g$ is nondecreasing, we have $g(\mathcal G\bar\omega)\geq g(\mathcal G\bar\omega-\bar\lambda)$ in $D$, thus \eqref{ba000} implies
\begin{equation}\label{ba001}
g(\mathcal G\bar\omega)=g(\mathcal G\bar\omega-\bar\lambda) \quad\mbox{ in } D.
\end{equation}
However, this is impossible, since on the set $\{x\in D\mid m<\mathcal G\bar\omega(x)<m+\bar\lambda\}$  we have 
\[g(\mathcal G\bar\omega-\bar\lambda)<g(m)\leq g(\mathcal G\bar\omega).\]
Here we used the fact that $g$ is strictly increasing on $(-\infty,m]$.
\end{proof}

Now for any $\omega\in\mathcal R_{\bar\omega}$ we choose  some $\bar\lambda=\bar\lambda(\omega)$ such that Lemma \ref{epex} holds.  Define 
\begin{equation}\label{hddf}
\hat{\mathcal D}(\omega)=\min_{\lambda\in\mathbb R}\mathcal D_{\lambda}(\omega)=-\frac{1}{2}\int_D\omega\mathcal G \omega dx +\int_DG(\mathcal G\omega-\bar\lambda(\omega)) dx+\bar\lambda(\omega) M_0.
\end{equation}
By Lemma \ref{ba67}, it is obvious that
\begin{equation}\label{ba90}
\hat{\mathcal D}(\bar\omega)=-\frac{1}{2}\int_D\bar\omega\mathcal G \bar\omega dx +\int_DG(\mathcal G\bar\omega) dx.
\end{equation}

The next two lemmas indicate that $\hat{\mathcal D}$ is a supporting functional of $EC$.
\begin{lemma}\label{cont0}
$\hat{\mathcal D}(\omega)\geq EC(\omega)$ for any $\omega\in \mathcal R_{\bar\omega}.$
\end{lemma}

\begin{proof}
It follows Lemma \ref{epex} and the definition of $\hat{\mathcal D}$.
\end{proof}

\begin{lemma}\label{cont}
$\hat{\mathcal D}(\bar\omega)=EC(\bar\omega).$
\end{lemma}

\begin{proof}
Since $\bar\omega=g(\mathcal G\bar\omega)$, we deduce from (ii) in Lemma \ref{lt} that \[G(\mathcal G\bar \omega)+\hat G(\bar \omega)=\bar\omega\mathcal G \bar\omega \,\,\mbox{ a.e. in }D.\] Taking into account \eqref{ba90}, we have 
\begin{align*}
\hat{\mathcal D}(\bar\omega)
&=-\frac{1}{2}\int_D\bar\omega\mathcal G \bar\omega dx +\int_DG(\mathcal G\bar\omega) dx\\
&=-\frac{1}{2}\int_D\bar\omega\mathcal G \bar\omega dx +\int_D\bar\omega\mathcal G\bar\omega-\hat G(\bar \omega)dx\\
&=\frac{1}{2}\int_D\bar\omega\mathcal G \bar\omega dx -\int_D\hat G(\bar \omega)dx\\
&=EC(\bar\omega).
\end{align*}
\end{proof}

Now we are ready to prove Lemma \ref{lem1}.
\begin{proof}[Proof of Lemma \ref{lem1}]
We need only to prove that $\bar\omega$ is an isolated local maximizer of $\hat{\mathcal D}$  on $\mathcal R_{\bar\omega}$. In fact, if this is true, then by Lemma \ref{cont0} and Lemma \ref{cont}, we deduce that $\bar\omega$ is an isolated local maximizer of $EC$  on $\mathcal R_{\bar\omega},$ hence $\bar\omega$ is an isolated local maximizer of $E$  on $\mathcal R_{\bar\omega}.$

Suppose by contradiction that $\bar\omega$ is not an isolated local maximizer of $\hat{\mathcal D}$  on $\mathcal R_{\bar\omega}$.  Then there exists a sequence $\{\phi_n\}_{n=1}^{+\infty}\subset\mathcal R_{\bar\omega}-\bar\omega$ such that $\phi_n\neq0$ for each $n$, $\|\phi_n\|_{L^2(D)}\to 0$ as $n\to+\infty$,
and
\begin{equation}\label{67yhh}
\hat{\mathcal D}(\bar\omega+\phi_n)\geq \hat{\mathcal D}(\bar\omega).
\end{equation}
Choose $\bar\lambda_n$ such that 
\begin{equation}\label{ba68}
\mathcal D_{\bar\lambda_n}(\bar\omega+\phi_n)=\hat{\mathcal D}(\bar\omega+\phi_n).
\end{equation}
From\eqref{ba90}, \eqref{67yhh} and \eqref{ba68} we obtain
\begin{equation}\label{67y}
\begin{split}&-\frac{1}{2}\int_D(\bar\omega+\phi_n)\mathcal G(\bar\omega+\phi_n) dx+\int_DG(\mathcal G\bar\omega+\mathcal G\phi_n-\bar\lambda_n)dx+\bar\lambda_nM_0\\
&\geq -\frac{1}{2}\int_D\bar\omega\mathcal G\bar\omega dx+\int_DG(\mathcal G\bar\omega).
\end{split}
\end{equation}
or equivalently,
\[\int_D\phi_n\mathcal G\bar\omega dx+\frac{1}{2}\int_D\phi_n\mathcal G\phi_n dx\leq \int_DG(\mathcal G\bar\omega+\mathcal G\phi_n-\bar\lambda_n)-G(\mathcal G\bar\omega) dx+\bar\lambda_nM_0.\]
Using Taylor's theorem we have
\begin{equation}\label{kpp}
\begin{split}
&\int_D\phi_n\mathcal G\bar\omega dx+\frac{1}{2}\int_D\phi_n\mathcal G\phi_n dx\\
 \leq &
\int_D g(\mathcal G\bar\omega)(\mathcal G\phi_n-\bar\lambda_n)dx 
+\frac{1}{2}\int_D g'(\mathcal G\bar\omega+\theta_n(\mathcal G\phi_n-\bar\lambda_n))(\mathcal G\phi_n-\bar\lambda_n)^2dx+\bar\lambda_nM_0,
\end{split}
\end{equation}
where $0\leq \theta_n\leq 1.$ Recalling the relation $\bar\omega=g(\mathcal G\bar\omega)$, we get from \eqref{kpp} that
\begin{equation}\label{kppp}
\int_D\phi_n\mathcal G\phi_n dx\leq 
\int_D g'(\mathcal G\bar\omega+\theta_n(\mathcal G\phi_n-\bar\lambda_n))(\mathcal G\phi_n-\bar\lambda_n)^2dx.
\end{equation}
Below for simplicity we use $o(1)$ to denote various quantities that goes to 0 \emph{uniformly} as $n\to+\infty$. 
Therefore taking into account the continuity of $g'$ we get from \eqref{kppp} that
\begin{equation}\label{kpppp}
\int_D\phi_n\mathcal G\phi_n dx\leq 
\int_D g'(\mathcal G\bar\omega)(\mathcal G\phi_n-\bar\lambda_n)^2dx+o(1)\int_D (\mathcal G\phi_n-\bar\lambda_n)^2dx.
\end{equation}

Now we claim that
\begin{equation}\label{zzz}
\bar\lambda_n=\frac{\int_Dg'(\mathcal G\bar\omega)\mathcal G\phi_n dx}{\int_Dg'(\mathcal G\bar\omega)dx}+o(1)\int_D|\mathcal G\phi_n |dx.
\end{equation}
To this end, first we show that 
\begin{equation}\label{ahah}
\bar\lambda_n=o(1).
\end{equation}
By Lemma \eqref{epex}, $\bar\lambda_n$ necessarily satisfies
\begin{equation}\label{bal2}
\int_Dg(\mathcal G\bar\omega+\mathcal G\phi_n-\bar\lambda_n) dx=M_0.
\end{equation}
Hence $\{\bar\lambda_n\}_{n=1}^{+\infty}$ must be a bounded sequence. 
Now suppose by contradiction that \eqref{ahah} is not true. Then there exists some subsequence $\{\bar\lambda_{n_j}\}_{j=1}^{+\infty}$ such that $\bar\lambda_{n_j}\to\lambda_0$ for some $\lambda_0\neq 0$. Taking into account the fact that \[\lim_{j\to+\infty}\|\mathcal G\phi_{n_j}\|_{L^\infty(D)}=0,\] 
which can be verified easily by standard elliptic estimate, and passing the limit $j\to+\infty$ in \eqref{bal2},   we get 
\[\int_Dg(\mathcal G\bar\omega-\lambda_0)dx=M_0,\]
which  contradicts Lemma \ref{ba67}.

We continue to prove \eqref{zzz}. From \eqref{bal2} and the fact that
\[\int_Dg(\mathcal G\bar\omega)dx=\int_D\bar\omega dx=M_0,\]
we have
\[\int_Dg(\mathcal G\bar\omega+\mathcal G\phi_n-\bar\lambda_n)-g(\mathcal G\bar\omega) dx=0.\]
Then the mean value theorem yields 
\[\int_Dg'(\mathcal G\bar\omega+\mu_n(\mathcal G\phi_n-\bar\lambda_n))(\mathcal G\phi_n-\bar\lambda_n)dx=0,\]
where $0\leq \mu_{n}\leq 1$. Thus
\begin{align*}
\bar\lambda_n&=\frac{\int_Dg'(\mathcal G\bar\omega+\mu_n(\mathcal G\phi_n-\bar\lambda_n))\mathcal G\phi_n dx}{\int_Dg'(\mathcal G\bar\omega+\mu_n(\mathcal G\phi_n-\bar\lambda_n))dx}\\
&=\frac{\int_Dg'(\mathcal G\bar\omega)\mathcal G\phi_n dx+o(1)\int_D|\mathcal G\phi_n| dx}{\int_Dg'(\mathcal G\bar\omega)dx+o(1)}\\
&=\frac{\int_Dg'(\mathcal G\bar\omega)\mathcal G\phi_n dx}{\int_Dg'(\mathcal G\bar\omega)dx}+o(1)\int_D|\mathcal G\phi_n| dx.
\end{align*}
Here the assumption $\int_Dg'(\mathcal G\bar\omega)dx>0$ was used. This proves the claim.

Inserting \eqref{zzz} into \eqref{kpppp}, we obtain after performing a simple calculation
\begin{equation}\label{djj}
\int_D\phi_n\mathcal G\phi_n dx-\int_Dg'(\mathcal G\bar\omega)(\mathcal G\phi_n)^2+o(1)\int_D(\mathcal G\phi_n)^2dx+\frac{\left(\int_Dg'(\mathcal G\bar\omega)\mathcal G\phi_n dx\right)^2}{\int_Dg'(\mathcal G\bar\omega)dx}\leq0.
\end{equation}
Since $\phi_n\neq 0$ for each $n,$  we deduce that $\int_D(\mathcal G\phi_n)^2dx>0$ for each $n$,
thus \eqref{djj} contradicts \eqref{deta}. This finishes the proof.

\end{proof}

\begin{remark}
The above method can also be used to simplify the proof of Theorem \ref{wg2}. However, for general $W_m$, after some attempts we find that the condition ``$g'$ has a positive lower bound on $[m,M]$" is still needed. Fortunately, Lemma \ref{lem1} is enough for us to prove Theorem \ref{thm1}.
\end{remark}
\begin{remark}\label{bbbttt}
If we require \eqref{deta} to hold for any $\phi\in L^2(D)$, then by the same argument we can prove that $\bar\omega$ is in fact an isolated maximizer of $EC$  in $L^2(D)$ (by (iii) in Lemma \ref{mM} $EC(\omega)$ is well defined for $\omega\in L^2(D)$). More precisely, there exists some $\epsilon>0,$ such that for any $w\in L^2(D),$ $0<\|w-\bar\omega\|_{L^2(D)}<\epsilon,$ it holds that $EC(\bar\omega)>EC(w).$
\end{remark}

\section{Proof of Theorem \ref{thm1}}
In this section, we give the proof of Theorem \ref{thm1}.
\begin{proof}[Proof of Theorem \ref{thm1}]
First we assume that $\int_Dg'(\mathcal G\bar\omega)dx=0.$ We claim that in this case $\bar\omega$ is a constant in $D$, and thus the conclusion of Theorem \ref{thm1} follows immediately. 
In fact, if $m=M$, then $\mathcal G\bar\omega$ is a constant in $D$, hence $\bar\omega \equiv 0.$ If $m<M$, then the open set
\[O=\{s\in(m,M)\mid g'(s)>0\}\] 
must be empty. In fact, if $O$ is not empty, then 
\[\{x\in D\mid g'(\mathcal G\bar\omega(x))>0\}=\{x\in D\mid \mathcal G\bar\omega(x)\in O\}\]
 is a nonempty open subset of $D$, which yields
$\int_Dg'(\mathcal G\bar\omega)dx>0$, a contradiction. Thus either $M=m=0$ or $g$ is a constant on $[m,M]$. In both cases $\bar\omega$ is a constant in $D$.

Below we assume that 
\begin{equation}\label{ne7}
\int_Dg'(\mathcal G\bar\omega)dx> 0. 
\end{equation}
By Lemma \ref{lem1}, it suffices to show that there exists some $\delta>0$ such that  
\begin{equation}\label{bbcbtt}
\int_D\phi\mathcal G\phi dx-\int_Dg'(\mathcal G\bar\omega)(\mathcal G\phi)^2 dx+\frac{\left(\int_Dg'(\mathcal G\bar\omega)\mathcal G\phi dx\right)^2}{\int_Dg'(\mathcal G\bar\omega)dx}\geq \delta \int_D(\mathcal G\phi)^2 dx,\,\,\forall\,\phi\in \mathcal R_{\bar\omega}-\bar\omega.
\end{equation}
 Below we prove a stronger statement, i.e., 
there exists some $\delta>0$ such that
\begin{equation}\label{bbbtt}
\int_D\phi\mathcal G\phi dx-\int_Dg'(\mathcal G\bar\omega)(\mathcal G\phi)^2 dx+\frac{\left(\int_Dg'(\mathcal G\bar\omega)\mathcal G\phi dx\right)^2}{\int_Dg'(\mathcal G\bar\omega)dx}\geq \delta \int_D(\mathcal G\phi)^2 dx,\,\,\forall\,\phi\in L^2(D).
\end{equation}

Suppose by contradiction that \eqref{bbbtt} is not true. Then there exists a sequence $\{\phi_n\}_{n=1}^{+\infty}\subset L^2(D)$ such  that
\begin{equation}\label{117}
\int_D\phi_n\mathcal G\phi_n dx-\int_Dg'(\mathcal G\bar\omega)(\mathcal G\phi_n)^2 dx+\frac{\left(\int_Dg'(\mathcal G\bar\omega)\mathcal G\phi_n dx\right)^2}{\int_Dg'(\mathcal G\bar\omega)dx}<\frac{1}{n}\int_D(\mathcal G\phi_n)^2 dx\quad \forall\, n.
\end{equation}
Obviously $\phi_n\neq 0$ for each $n$.
Denote \[\zeta_n=\frac{\phi_n}{\|\mathcal G\phi_n\|_{L^2(D)}}.\] Then for each $n$ it holds that
\begin{equation}\label{115}
\|\mathcal G\zeta_n\|_{L^2(D)}=1
\end{equation}
 and 
\begin{equation}\label{1172}
\int_D\zeta_n\mathcal G\zeta_n dx-\int_Dg'(\mathcal G\bar\omega)(\mathcal G\zeta_n)^2 dx+\frac{\left(\int_Dg'(\mathcal G\bar\omega)\mathcal G\zeta_n dx\right)^2}{\int_Dg'(\mathcal G\bar\omega)dx}<\frac{1}{n}.
\end{equation}
By the condition \eqref{sscv} in Theorem \ref{thm1}, 
we have
\begin{equation}\label{118}
\int_D\zeta_n\mathcal G\zeta_n dx-\int_Dg'(\mathcal G\bar\omega)(\mathcal G\zeta_n)^2 dx\geq 0,\quad\forall\,n.
\end{equation}
Hence \eqref{1172} and \eqref{118} together yield 
\begin{equation}\label{119}
\int_D\zeta_n\mathcal G\zeta_n dx-\int_Dg'(\mathcal G\bar\omega)(\mathcal G\zeta_n)^2 dx\to 0
\end{equation}
and 
\begin{equation}\label{121}
\int_Dg'(\mathcal G\bar\omega)\mathcal G\zeta_ndx\to 0
\end{equation}
as $n\to+\infty$.

By \eqref{115} and \eqref{119}, we see that  $\{\mathcal G\zeta_n\}_{n=1}^{+\infty}$ is a bounded sequence in $H^1_0(D).$ Up to a subsequence, we assume that $\mathcal G\zeta_n\rightharpoonup \tilde u$  in $H^1_0(D)$ for some $\tilde u\in H^1_0(D)$. Then it follows from \eqref{115} that
\begin{equation}\label{1120}
\|\tilde u\|_{L^2(D)}=1,
\end{equation} 
and from \eqref{121} that 
\begin{equation}\label{1123}
\int_Dg'(\mathcal G\bar\omega)\tilde u dx=0.
\end{equation}

On the other hand, it is easy to see that
\[0\leq \int_D|\nabla \tilde u|^2 dx-\int_Dg'(\mathcal G\bar\omega)\tilde u^2 dx\leq \liminf_{n\to+\infty}\left(\int_D\phi_n\mathcal G\zeta_n dx-\int_Dg'(\mathcal G\bar\omega)(\mathcal G\zeta_n)^2 dx\right)= 0,\]
thus 
\[ \int_D|\nabla \tilde u|^2 dx-\int_Dg'(\mathcal G\bar\omega)\tilde u^2 dx= 0,\]
which means $\tilde u$ minimizes the following quadratic functional
\[\int_D|\nabla u|^2 dx-\int_Dg'(\mathcal G\bar\omega)u^2 dx\]
in $H^1_0(D)$, or equivalently, $\tilde u$ is the first eigenfunction of the operator $-\Delta-g'(\mathcal G\bar\omega)$.
Since $g\in C^{1,\alpha}[m,M]$, we see that $g'(\mathcal G\bar\omega)\in C^\alpha(\bar D)$. Moreover, by \eqref{1120} $\tilde u\neq0.$ 
Therefore we can use Lemma \ref{fefu} to get $u>0$ or $u<0$ in $D$. This contradicts \eqref{ne7} and \eqref{1123}.

\end{proof}

\begin{remark}
By Remark \ref{bbbttt} and \eqref{bbbtt}, the $\bar\omega$ in Theorem \ref{thm1} is in fact an isolated maximizer of $EC$ in $L^2(D).$

\end{remark}
\bigskip

{\bf Acknowledgements:}
{G. Wang was supported by National Natural Science Foundation of China (12001135, 12071098) and China Postdoctoral Science Foundation (2019M661261, 2021T140163).}

\phantom{s}
 \thispagestyle{empty}

\end{document}